\documentclass[11pt]{amsart}
\usepackage{amsmath, amssymb, amsfonts, amstext, verbatim, amsthm, mathrsfs}
\usepackage[margin=3.5cm,marginpar=2.8cm]{geometry}
\usepackage{graphicx}
\usepackage{amssymb}
\usepackage{epstopdf}
\usepackage[usenames]{color}
\usepackage{marginnote}

\usepackage{microtype}
\usepackage[all]{xy}
\usepackage[modulo]{lineno}
\usepackage[usenames]{color}
\usepackage{aliascnt}
\usepackage{enumitem}
\usepackage{xspace}
\usepackage{amsfonts}
\usepackage{amssymb}
\usepackage[centertags]{amsmath}
\usepackage{amsthm}
\usepackage{rotating}
\usepackage{dsfont}
\usepackage{bm}
\usepackage{color}
\usepackage{subfigure}
\usepackage{amsmath}
\usepackage{array}
\usepackage[all]{xy}
\usepackage{euscript}
\usepackage[T1]{fontenc}
\usepackage{mathbbol}
\usepackage{hyperref}
\usepackage{fancyhdr}

\theoremstyle{plain}
\newtheorem{theorem}{Theorem}[section]
\newtheorem{lemma}[theorem]{Lemma}
\newtheorem{corollary}[theorem]{Corollary}
\newtheorem{prop}[theorem]{Proposition}

\newtheorem{question}[theorem]{Question}

\theoremstyle{definition}
\newtheorem{remark}[theorem]{Remark}
\newtheorem{example}{Example}
\newtheorem{disclaimer}[theorem]{Disclaimer}

\theoremstyle{remark}

\newcommand{\gapac}{\mathfrak{g}}
\newcommand{\C}{\mathbb{C}}
\newcommand{\Z}{\mathbb{Z}}
\newcommand{\bb}{\frak{b}}

\newcommand{\ovl}{\overline}

\renewcommand{\lll}{\Langle}
\newcommand{\rrr}{\Rangle}
\newcommand{\T}{\mathcal{T}}
\newcommand{\wh}{\widehat}
\newcommand{\bdy}{\partial}
\newcommand{\R}{\mathbb{R}}

\newcommand{\op}[1]{{\operatorname{#1}}}
\renewcommand{\bar}{\mathcal{B}}

\newcommand{\hooksymp}{\overset{s}\hookrightarrow}
\newcommand{\aug}{\epsilon}
\newcommand{\TT}{\T^{\bullet}}

\newcommand{\calA}{\mathcal{A}}
\newcommand{\Q}{\mathbb{Q}}
\newcommand{\sixty}{100}
\newcommand{\CP}{\mathbb{CP}}
\newcommand{\eps}{\varepsilon}
\newcommand{\Li}{\mathcal{L}_\infty}
\newcommand{\can}{\op{can}}

\newcommand{\CC}{{\mathbb C}}

\let\oldtocsection=\tocsection
\let\oldtocsubsection=\tocsubsection

\renewcommand{\tocsection}[2]{\hspace{0em}\oldtocsection{#1}{#2}}
\renewcommand{\tocsubsection}[2]{\hspace{1em}\oldtocsubsection{#1}{#2}}

\addtocontents{toc}{\protect\setcounter{tocdepth}{2}}

\title{Higher symplectic capacities and the stabilized embedding problem for integral elllipsoids}

\author{Dan Cristofaro-Gardiner, Richard Hind and Kyler Siegel}

\begin{document}
\maketitle
\pagestyle{plain} 

\begin{abstract}{
The third named author has been developing a theory of ``higher'' symplectic capacities.  These capacities are invariant under taking products, and so are well-suited for studying the stabilized embedding problem.  The aim of this note is to apply this theory, assuming its expected properties, to solve the stabilized embedding problem for integral ellipsoids, when the eccentricity of the domain has the opposite parity of the eccentricity of the target and the target is not a ball.  For the other parity, the embedding we construct is definitely not always optimal; also, in the ball case, our methods recover previous results of McDuff, and of the second named author and Kerman. 
There is a similar story, with no condition on the eccentricity of the target, when the target is a polydisc: a special case of this implies a conjecture of the first named author, Frenkel, and Schlenk concerning the rescaled polydisc limit function.  Some related aspects of the stabilized embedding problem and some open questions are also discussed.


} 
\end{abstract}



\tableofcontents

\section{Introduction}

\subsection{The main results} 

Let $X_1$ and $X_2$ be four-dimensional symplectic manifolds.    There has recently been considerable interest in understanding the {\bf stabilized symplectic embedding problem}, namely the question of whether or not there exists a symplectic embedding
\begin{equation}
\label{eqn:stabprob}
X_1 \times \C^N \hooksymp X_2 \times \C^N.
\end{equation}
Indeed, certain techniques which are available for studying four-dimensional embedding problems do not have a clear analogue in higher dimensions, and so it is interesting to understand how different the stabilized problem is from the four-dimensional one.  For more about the stabilized problem, we refer the reader to \cite{Hind-Kerman_new_obstructions,Hind-Kerman_correction,CG-Hind_products,Ghost_stairs_stabilize,Mcduff_remark}, the references therein, and the discussion below. 

The embedding problem \eqref{eqn:stabprob} is already quite subtle when $X_1$ and $X_2$ are simple shapes, like {\bf ellipsoids}
\[ E(a,b) := \left\lbrace \frac{ \pi |z_1|^2}{a} + \frac{\pi |z_2|^2}{b} \le 1 \right\rbrace \subset \C^2,\]
{\bf balls} $B(c) := E(c,c),$ {\bf polydiscs}
\[ P(a,b) := \left\lbrace \frac{ \pi |z_1|^2}{a} \le 1, \frac{\pi |z_2|^2}{b} \le 1 \right\rbrace \subset \mathbb{C}^2,\]
and {\bf cubes} $C(c) := P(c,c).$  
(Here, $\mathbb{C}^N$ is equipped with its standard symplectic form.)  For example, what is known about the stabilized ellipsoid-into-ball problem has a curious mix of rigidity and flexibility: much about this question remains unknown.  In contrast, the stabilized polydisc-into-ball problem is completely solved \cite[Thm. 1.3.5]{chscI} (for another approach see \cite{hist}) and the answer is described by a very simple function, namely a piecewise linear function with two pieces. 


The starting point for our investigations here is the stabilized ellipsoid-into-ellipsoid problem. This is a special case of Problem $44$ in the influential problem list \cite[Ch. 14]{McDuff-Salamon} by McDuff and Salamon, which asks for a solution to the symplectic embedding problem for $2n$-dimensional symplectic ellipsoids:   we can view stabilized ellipsoids as $2n$ dimensional ellipsoids with most arguments infinite.  Consider the function $c^N_{b,ell}(a)$, defined to be the infimum, over $\lambda$, such that there exists an embedding
\begin{equation}\label{eqn:stabilizedell}
E(1,a) \times \C^N \hooksymp \lambda \cdot E(1, b) \times \C^N,
\end{equation}
where we write $\lambda \cdot E(a, b)$ for $E(\lambda a, \lambda b)$.  
This function for $a, b \ge 1$ completely determines the stabilized ellipsoid-into-ellipsoid problem, and we would ideally like to compute it.
At present, this looks out of reach.  As mentioned above, even the case $b = 1$ seems quite subtle; in fact it is the focus of a conjecture by McDuff \cite{Mcduff_remark}.  And, when $b > 1$, almost nothing is currently known.  However, it turns out that when $a$ and $b$ are integers, there is a lot more traction.

\begin{theorem}
\label{thm:main2}
Assume that $b > 1$ is an integer, and let $a \ge b+1$ be any integer with parity the opposite of $b$.  Then for $N \ge 1,$
\[ c^N_{b,ell} (a) = \frac{2a}{a+b-1}.\]

\end{theorem}

We discuss the 
hypothesis $a \ge b+1$ here in Section~\ref{firststep} below, where we show that it is essentially necessary.



A key aspect of our proof of the above theorem, which is one of the motivations for writing this note, involves the obstructions required to prove it.  Symplectic embedding problems are profitably studied by {\bf symplectic capacities}, see e.g. \cite{q}.   The third named author has recently defined a new sequence of symplectic capacities $\gapac_k$ which play a starring role here.  These capacities $\gapac_k$ are invariant under taking products with $\C$ and so give obstructions to the stabilized problem.  As we will see in the proof of Theorem~\ref{thm:main2}, the $\gapac_k$ are very well-adapted to proving Theorem~\ref{thm:main2}, and the obstructive side of the proof follows quite quickly once we can marshal them to our benefit.   The constructive side of the proof comes from a variant of the stabilized folding construction pioneered by the second named author.  

\begin{disclaimer}

Our high level discussion of symplectic capacities in \S\ref{sec:new_capac} follows \cite{hsc}, which in turn assumes the existence of rational symplectic field theory with its expected functoriality properties as outlined in \cite{eliashberg2000introduction}. Apart from simple special cases, such a formalism is known to require a virtual perturbation framework such as the theory of polyfolds;
for the current status of this and related projects we refer the reader to e.g. \cite{hofer2017polyfold,fish2018lectures, pardon2019contact,HuN, bao2015semi, ishikawa2018construction} and the references therein.

The proofs of our main results on embedding obstructions in \S\ref{sec:computations} take the properties of the capacities $\gapac_k$ summarized in Theorem~\ref{thm:main_props} as a black box, together with some computations from \cite{hsc} which we recall in \S\ref{subsubsec:some_obstr}.
Our proof of Theorem~\ref{thm:main2} furthermore requires the formula for $\gapac_k(E(1,a))$ which will appear in the forthcoming work \cite{McDuffSiegel_in_prep}. The latter reference also constructs an ersatz version of these capacities in the special case of ellipsoids without appealing to virtual perturbations; these give equivalent obstructions for stabilized embeddings between four-dimensional ellipsoids, and the method also readily adapts to the case of ellipsoid domain and polydisk target.
Our proof of Proposition~\ref{prop:comp} further depends on the formalism of \cite{chscI}, which is based on \cite{hsc} and the forthcoming \cite{chscII}.

\end{disclaimer}

In dimension four, when $b$ is integral there is an equivalence of embeddings
\begin{equation}
\label{eqn:equiv}
E(1,a) \hooksymp \lambda P(1,b), \quad E(1,a) \hooksymp \lambda E(1,2b),
\end{equation}
that is, one of these embeddings exists if and only if the other does, see for example \cite[Rmk. 1.2.1]{cristofaro2017symplectic}.  So, it is natural to compare Theorem~\ref{thm:main2}  with the stabilized ellipsoid-into-polydisc problem.  Here we get a somewhat parallel, but in fact stronger result.  Define  $c^N_{b,poly}(a)$ to be the infimum, over $\lambda$, such that an embedding
\begin{equation}
\label{eqn:stabilizedone} E(1,a) \times \C^N \hooksymp \lambda \cdot P(1, b) \times \C^N 
\end{equation}
exists.

\begin{theorem}
\label{thm:main}
Let $a \ge 2b-1$ be any odd integer.  Then for $N \ge 1,$
\[ c^N_{b,poly}(a) =  \frac{2a}{a+2b-1}.\]
\end{theorem}

We remark that, in contrast to Theorem~\ref{thm:main2}, there is no requirement here that $b$ is an integer.   As with the previous theorem, the hypothesis $a \ge 2b-1$ is discussed in Section~\ref{firststep}, where it is shown to be necessary.

\subsection{Applications and remarks}
\label{sec:app}

\subsubsection{Steps and the rescaled embedding function.}\label{steps}

One of our motivations for studying Theorem~\ref{thm:main} is that it readily implies a conjecture of the first author, Frenkel, and Schlenk about the stabilized ellipsoid-into-polydisc function, namely Conjecture 1.4 in \cite{cristofaro2017symplectic}, which we now explain.  

First we explain the motivation behind that conjecture.  As alluded to above, at present, fully computing the function $c^N_{b,poly}(a)$ for $N \ge 1$ seems quite difficult.  However, there is a related function, called the {\bf rescaled limit function} $\hat{c}^N_{b,poly}$, see \eqref{eqn:resc} below, that looks more tractable and in particular could be computed given a resolution of the aforementioned Conjecture 1.4. 

To elaborate,
the function $c_{b,poly}^0(a)$ for $b \in \Z_{\geq 2}$ was previously computed by the first author, Frenkel and Schlenk in \cite{cristofaro2017symplectic}.  
It was shown that the function $c^0_{b,poly}(a)$ is given by the volume constraint $\sqrt{\frac{a}{2b}}$, except on finitely many intervals.  On all but one of these intervals, the function $c^0_{b,poly}(a)$ is given by a ``linear step'': it is piecewise linear, with a single nonsmooth point, called its corner, where its graph changes from lying on a line through the origin to being horizontal.  On the remaining interval, it is also piecewise linear with a single nonsmooth point, but the linear piece does not lie on a line through the origin -- it has an intercept, and so we call it the ``affine step''.  For more detail, see \cite{cristofaro2017symplectic}.

Conjecture $1.4$ asserts that the linear steps from above are ``stable''.  Of course,
for any $a$, we have $c^N_{b,poly}(a) \le c^0_{b,poly}(a),$ by taking the product with the identity mapping.  
The conjecture, then, is that for $a$ in the domain of the linear steps, we have $c^N_{b,poly}(a) = c^0_{b,poly}(a).$  To state that conjecture precisely, we define, for $k \in \lbrace 0, 1, 2, \ldots, \lfloor \sqrt{2 b} \rfloor \rbrace,$ the numbers
\[ u_b(k) = \frac{(2b+k)^2}{2b}, \quad \quad v_b(k) = 2b \left(\frac{2b+2k+1}{2b+k}\right)^2. \]
We always have $u_b(k) < v_b(k)$ except if $k^2 = 2b$; for $u_b(k) < a < v_b(k)$, the graph of $c^N_{b,poly}(a)$ is precisely the linear steps mentioned above. 

\begin{corollary}[Conj. 1.4 of \cite{cristofaro2017symplectic}]
\label{cor:main}
Assume that $b$ is an integer and 
\[u_b(k) \le a \le v_b(k).\]
Then 
\[c^0_{b,poly}(a) = c^N_{b,poly}(a) = c^0_{2b,ell}(a) = c^N_{2b,ell}(a).\]
\end{corollary}

The final two equalities here, concerning the ellipsoid-into-ellipsoid function, were not actually part of Conjecture 1.4; however, they fall out immediately from our proof.

We now state the relevance of this to the rescaled limit function.  The background is that \cite{cristofaro2017symplectic} defined\footnote{Actually, only the $N = 0$ case of these functions was defined, but the definition extends verbatim to general $N$, and that will be our working definition here.} the rescaled functions
\begin{equation}
\label{eqn:resc}
 \hat{c}_{b,poly}^N(a) := 2b c^N_{b,poly}(a+ 2b) - 2b, \quad a \ge 0, 
 \end{equation}
 in order to capture the qualitative behavior of the obstructive part of the embedding function $c_{b,poly}^0$ that goes beyond Gromov's nonsqueezing theorem.  It was shown in \cite[Eq. 1.3]{cristofaro2017symplectic} that the functions $\hat{c}^0_{b,poly}(a)$ converge, as $b \to \infty$, uniformly on bounded sets to a pleasing answer, namely the ``infinite regular staircase'' described by the function $c_{\infty}(a): [0,\infty) \to \R$ whose graph consists of infinitely many linear steps of width 2, see \cite[Fig. 1.7]{cristofaro2017symplectic} and Figure~\ref{fig:figure} below.  For more about the motivation for studying the rescaled function, we refer the reader to the discussion in
\cite[Sec. 1.2]{cristofaro2017symplectic}.

\begin{figure}[h!]
\centering
\includegraphics[width=.5\textwidth]{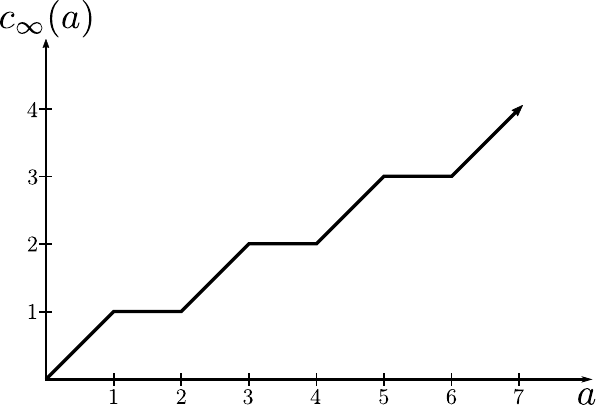}
\caption{The rescaled limit function.  Each step has width two, and consists of a line of slope one and a horizontal line.}
\label{fig:figure}
\end{figure}

\begin{corollary}
\label{cor:rescal}
The rescaled limit function is stable.  That is, for any $N \in \Z_{\geq 0}$ and integral $b$, we have
\[ \lim_{b \to \infty} \hat{c}^N_{b,poly}(a) = c_{\infty}(a), \quad a \in [0,\infty)\]
uniformly on bounded sets.
\end{corollary}  

We will explain the proofs of these corollaries in \S\ref{sec:cor}.

\subsubsection{The first step}\label{firststep}

We next remark that, in the context of Theorem~\ref{thm:main2}, the lower bound on $a$ is essentially necessary. Indeed, if $a \le b$ then inclusion gives an embedding which Gromov's non-squeezing theorem shows is optimal.  That is, $c^N_{b,ell} (a) = 1$ for all $N \ge 0$.  There is a similar story for Theorem~\ref{thm:main} for $a \le 2b-1$, but it requires a more interesting embedding.  With a little more work, we can extend the range of $a$ to work out at least part\footnote{In fact, Proposition~\ref{prop:main} likely describes the entirety of the first step, although we do not address this here.}  of the ``first step" of the embedding functions considered in this note.   

\begin{prop}
\label{prop:main} 
Let $b \in \mathbb{R}_{\ge 1}.$  Then:
\begin{itemize}
\item 
The function $c^N_{b, ell}$ starts as follows:
\begin{itemize}
\item We have $c^N_{b,ell}(a) = 1, \quad 1 \le a \le b.$
\item We have $c^N_{b,ell}(a) = \frac{a}{b}, \quad b \le a \le \lfloor b \rfloor +1.$
\end{itemize}
\item The function $c^{N}_{b,poly}$ starts as follows.  Let $a_0$ be the smallest odd integer that is no less than $2b-1$.
\begin{itemize}
\item We have  $c^{N}_{b,poly}(a) = 1, \quad 1 \le a \le \frac{a_0-1}{2} + b.$
\item We have $ c^{N}_{b,poly}(a) = \frac{2}{a_0 + 2b-1} a, \quad \frac{a_0-1}{2} + b \le a \le a_0$
\end{itemize}
\end{itemize}
\end{prop}

Note that there is no restriction above that $a$ or $b$ be integral, in contrast to the theorems in the previous section.

\subsubsection{The case $b = 1$} 

In view of Theorem~\ref{thm:main2}, it is natural to ask about the case $b = 1$.  This was previously studied by McDuff \cite{Mcduff_remark}, who proved an analogous result for any integer congruent to two, modulo three; we can recover this result with our methods as well, see Example~\ref{ex:mcd} in \S\ref{sec:computations1}.  Comparing our result to McDuff's, it is interesting to note the switch from three periodicity to two periodicity as $b$ increases from one.  There is a substantial mystery about the structure as $b$ ranges from $1$ to $2$,  see \S\ref{sec:b12}, which we plan to investigate in follow-up work. 

\subsubsection{The other parity}

In view of the above results,  it is natural to ask: what happens for $a$ an integer of a parity not covered by our theorems.  We certainly do not have a satisfactory answer to this at present.  However, using the more general calculus of \cite{chscI}, together with the aid of the computer, 
we can show for example:

\begin{prop}
\label{prop:comp}
For $6 \le a \le \sixty$ an even integer, the conclusion of Theorem~\ref{thm:main2} holds for $b = 2$, that is for $N \ge 1$ we have 
\[ c^N_{2,ell}(a) = \frac{2a}{a+1}.\]
Similarly, for $6 \le a \le \sixty$ an even integer, the conclusion of Theorem~\ref{thm:main} holds for $b = 1$, that is for $N \ge 1$ we have
\[ c^N_{1,poly}(a) = \frac{2a}{a+1}.\] 
\end{prop} 
\begin{remark}
The assumption $a \geq 6$ in Proposition~\ref{prop:comp} is necessary. Indeed, for $a$ less than the squared silver ratio $\sigma^2 \approx 5.83$, $c^0_{1,poly}(a)$ is an infinite staircase \cite{fm}. In particular, we have $c^N_{1,poly}(a) \leq c^0_{1,poly}(a)$, and $c^0_{1,poly}(a)$ is strictly less than $\tfrac{2a}{a+1}$ for $a = 2,4$. 
The same applies for $c^N_{2,ell}$, since we have $c^0_{2,ell} = c^0_{1,poly}$.

For more examples, suppose that $a =2b+2k+2$ is an even integer. Referring to section \S\ref{steps} we see that $v_b(k) \le a \le u_b(k+1)$ which for $k \ge 2$ implies that $c^0_{b,poly}(a) = \sqrt{\frac{a}{2b}}$, that is, there is a volume filling embedding from $E(1,a)$ into a scaling of $P(1,b)$ (the point $a=2b+4$ lies in the affine step). By \eqref{eqn:equiv} this is equivalent to the existence of a volume filling embedding from $E(1,a)$ into a scaling of $E(1,2b)$. Now, volume filling embeddings in dimension $4$ improve on the folding construction giving Theorem~\ref{thm:main2} when $a < b+1 + 2\sqrt{b}$. Hence the conclusion of Theorem~\ref{thm:main2} is false when $a$ and $b$ are even and $b+4 < a < b+1 + 2\sqrt{b}$.

\end{remark}

 \subsection*{Structure of the note}
 
 In \S\ref{sec:new_capac} we review the construction of the higher symplectic capacities of the third named author; our discussion here includes some informal elements to help convey the intuition.  Then in \S\ref{sec:computations} we give the proofs of our results.   The final section \S\ref{sec:discuss} discusses some natural follow-up questions to this work.
 
 
 \subsection*{Acknowledgments}
 
 We thank Felix Schlenk for his encouragement, and for helping the first and third named authors better understand constructions of embeddings between stabilized ellipsoids.  We would also like to thank the referee for carefully reading our paper and for many useful comments.
 
Our paper is dedicated to Claude Viterbo on the occasion of his $60^{th}$ birthday.  We are immensely grateful to Claude for his visionary leadership of our field. 
 
This research was completed while the first named author was on a von Neumann fellowship at the Institute for Advanced Study; he thanks the Institute for their support. The first named author is partially supported by NSF grant DMS-1711976 and the second named author by Simons Foundation Grant no. 633715.

\section{New capacities}\label{sec:new_capac}

We first briefly review the capacities $\gapac_k$ defined for $k \in \Z_{\geq 1}$ in \cite{hsc}.  These are part of a more general family of capacities $\gapac_\bb$ indexed by elements in the symmetric tensor algebra $\ovl{S}\Q[t] = \bigoplus_{k =1}^\infty (\otimes^k \Q[t])/\Sigma_k$. We give here only an impressionistic sketch, omitting some of the more technical details.
In addition to the computations described in \S\ref{subsubsec:some_obstr}, the key structural properties we will need are summarized in the following:
\begin{theorem}\label{thm:main_props}\cite{hsc}
For any Liouville domain $X$ and $k \in \Z_{\geq 1}$, we have $\gapac_k(X) \in \R_{> 0}$ with the following properties:
\begin{enumerate}
\item {\em symplectomorphism invariance}: if $X'$ is another Liouville domain which is symplectomorphic to $X$, we have $\gapac_k(X) = \gapac_k(X')$
\item {\em scaling}: if $X'$ is the Liouville domain obtained by scaling the Liouville form of $X$ by a constant $c  \in \R_{> 0}$, we have $\gapac_k(X') = c\gapac_k(X)$
\item {\em monotonicity}: if $X'$ is another Liouville domain and there exists a symplectic embedding $X \hooksymp X'$, then we have $\gapac_k(X) \leq \gapac(X')$
\item {\em stabilization}: we have $\gapac_k(X \times B^2(S)) = \gapac_k(X),$ provided that $S > \gapac_k(X)$.\footnote{Strictly speaking, $X \times B^2(S)$ is not a Liouville domain since it has corners, although these can be removed by an arbitrarily small smoothing. See \cite[\S5.4]{hsc} for a more precise formulation. Property (1) is of  course  automatic given property (3).}
\end{enumerate}
\end{theorem}
\noindent Note that (3) actually implies (1).

\subsection{The first approximation}\label{subsec:first_approx}

Suppose that $X$ is a Liouville domain.
We work with almost complex structures $J$ on the symplectic completion $\wh{X}$ which are admissible in the sense of symplectic field theory (SFT). 
Fix a point $p \in X$ along with a local $J$-holomorphic divisor $D$ passing through $p$.
To first approximation, $\gapac_k(X)$ is simply the minimal energy of a punctured $J$-holomorphic sphere $u: \Sigma \rightarrow \wh{X}$ with some number $l \geq 1$ of positive ends asymptotic to Reeb orbits in $\bdy X$, such that $u$ passes through $p$ and is tangent to $D$ to order $k-1$. 
We denote this tangency constraint by $\lll \T^{k-1}p\rrr$ (see \cite{McSie} and the references therein for more details).

To see why this should be monotone with respect to symplectic embeddings, 
the basic point is that given such a curve $u$ in $\wh{X}$ and a symplectic embedding $X' \hooksymp X$,
we can neck-stretch along $\bdy X'$. This forces $u$ to break into a pseudoholomorphic building consisting of
\begin{itemize}
\item
a curve $u_{\op{top}}$ (possibly disconnected) in the completed symplectic cobordism $\wh{X \setminus X'}$ with the same positive asymptotics as $u$
\item
a curve $u_{\op{bot}}$ in $\wh{X'}$ which inherits the tangency constraint $\lll \T^{k-1}p\rrr$. 
\end{itemize}
Since $u_{\op{bot}}$ is a candidate minimizer for $\gapac_k(X')$ and it has energy at most that of $u$, this shows that $\gapac_k(X') \leq \gapac_k(X)$. 

\subsection{Behavior under stabilization}

One role of the local tangency constraint in the definition of $\gapac_k$ is to cut down the dimension of familes of curves, thereby giving access to curves of higher Fredholm index. There are certainly other natural geometric constraints which lower the index, the most obvious being to impose $k$ distinct point constraints. In fact, doing so leads to the ``rational symplectic field theory capacities'' (RSFT) first considered in \cite{hutchings_2013}. 

However, point constraints behave in a rather complicated way under dimensional stabilization. The RSFT capacities are therefore 
perhaps not well-suited for stabilized problems (although they may have other applications yet to be discovered).  For example, note that each point constraint is codimension $2$ when $\dim X = 4$, but is generally codimension $2n-2$ when $\dim X = 2n$. This means that the same curve with the same point constraints has negative total index after stabilizing by $\mathbb{C}^N$ with $N$ large enough.

By contrast, local tangency constraints behave quite well with respect to stabilization. 
This is closely related to the observation of Hind and Kerman from \cite{Hind-Kerman_new_obstructions} that punctured rational curves with exactly one negative end have stable Fredholm index. 
The stabilization property in Theorem~\ref{thm:main_props} is also closely related to the stabilization theorems appearing in the works
\cite{CG-Hind_products,Ghost_stairs_stabilize,Mcduff_remark}.


\subsection{The naive chain complex}

Unfortunately, the definition given in \S\ref{subsec:first_approx} is not particularly robust, since it might depend on the choice of almost complex structure $J$.
Indeed, if we try to deform $J$ to some other almost complex structure $J'$, somewhere along the way the curve $u$ might degenerate into a pseudoholomorphic building and then disappear.
Therefore, in order to get something which is truly a symplectomorphism invariant, we have to be a bit more ``homological''.
This is where the chain complexes coming from Floer theory or symplectic field theory become essential.

The idea is to associate to $X$ a filtered chain complex $C(X)$, where
\begin{itemize}
\item
as a vector space, $C(X)$ is the (graded) polynomial algebra on the (not necessarily primitive) Reeb orbits of $\bdy X$
\item 
the differential is defined by counting rigid-up-to-translation connected rational curves in $\R \times \bdy X$ with several positive ends and one negative end
\item the filtration is by the symplectic action functional, or equivalently by the periods of Reeb orbits. 
\end{itemize}
Similarly, given an exact\footnote{There is also a nice story extending the theory to non-exact symplectic cobordisms, but we will ignore this for simplicity.} symplectic cobordism $W$ with positive end $\bdy^+W = \bdy X$ and negative end $\bdy^-W = \bdy X'$, we define a chain map from $C(X)$ to $C(X')$
by counting rigid possibly disconnected rational curves in $W$, such that each component has several positive ends and one negative end.
By Stokes' theorem, both the differential and the cobordism map are action-nondecreasing and hence preserve the filtrations.

However, the above prescription does not work on face value due to transversality issues. 
Namely, in order to show that the differential squares to zero and that the cobordism map is a chain map, the typical strategy is to analyze analogous moduli spaces of dimension one and show that (after compactifying) their boundaries give precisely the desired relations. 
But it is well-known that the relevant SFT moduli spaces are rarely transversely cut out for any choice of generic $J$. 
Multiply covered curves tend to appear with higher-than-expected dimension, and this spoils our strategy.

\subsection{Input from symplectic field theory}
One way is get around this issue is to count curves in a ``virtual'' sense, by introducing suitable abstract perturbations which allow more room to achieve transversality.
This is the basic strategy being pursued to define SFT in full generality by various groups, with much recent progress but consensus not yet achieved (see e.g. \cite{hofer2017polyfold,fish2018lectures, pardon2019contact,HuN, bao2015semi, ishikawa2018construction} and the references therein).


In the setting of SFT, the desired invariant $C(X)$ can be written as $\bar CH_{\op{lin}}(X)$. Here $CH_{\op{lin}}(X)$ is the linearized contact homology of $X$,
which is roughly the chain complex generated by Reeb orbits of $\bdy X$ with differential counting cylinders in the symplectization $\R \times \bdy X$.\footnote{More precisely, we only allow ``good'' Reeb orbits, and we count cylinders which are additionally ``anchored'' in $X$.}
 Linearized contact homology only involves curves with one positive end, but by incorporating curves with several positive ends we get an $\mathcal{L}_\infty$ structure, consisting of $l$-to-$1$ operations for all $l \geq 1$ satisfying various compatibility conditions.
We can conveniently package this $\mathcal{L}_\infty$ structure into one large chain complex $\bar CH_{\op{lin}}(X)$, the {\em bar complex}.


\subsection{From spectral invariants to capacities}

Getting back to the high level viewpoint, we have a filtered chain complex $C(X)$ for each Liouville domain $X$, and filtration-preserving chain maps $
\Xi: C(X) \rightarrow C(X')$ for any (exact) symplectic embedding $X' \hooksymp X$.
Now for any class $\alpha$ in the homology of $C(X)$, define $c_{\alpha}(X)$ to be the minimal action of any closed element of $C(X)$ which represents $\alpha$.
By a simple diagram chase, we have $c_{[\Xi](\alpha)}(X') \leq c_{\alpha}(X)$, where $[\Xi]$ denotes the homology-level map induced by $\Xi$.

At first glance, this construction appears to give a new family of symplectic capacities indexed by homology classes of $C(X)$. But there is still one issue,
which is that we need a canonical way to reference these homology classes. 
Indeed, in principle the homology level map $[\Xi]$ might be quite nontrivial, so how do we know when two numbers $c_{\alpha}(X)$ and $c_{\beta}(X')$ can be compared to each other?

This is where the tangency constraints come in.
The claim is that by counting possibly disconnected curves in $\wh{X}$ with each component $u_i$ satisfying a $\lll \T^{k_i-1}p\rrr$ constraint for some $k_i \in \Z_{>0}$, we get a chain map $$\aug_X\lll \TT\rrr: C(X) \rightarrow \ovl{S}\Q[t].$$
For example, a term $t^3 \odot t^2 \odot t^5$ in $\ovl{S}\Q[t]$ corresponds to counting curves with three components which satisfy constraints $\lll \T^{3}p\rrr$, $\lll \T^2p\rrr$, and $\lll \T^5p\rrr$ respectively.
Moreover, these maps are natural in the sense that the composition $\aug_{X'}\lll \TT\rrr  \circ \Xi$ agrees with $\aug_{X}\lll \TT \rrr$ up to filtered chain homotopy.

Now for any $\bb \in \ovl{S}\Q[t]$, we define the capacity $\gapac_\bb(X) \in \R_{> 0}$ by
$$\gapac_\bb(X) := \inf \{c_\alpha(X)\;:\; [\aug_X\lll \TT\rrr](\alpha) = \bb\}.$$
This defines a symplectomorphism invariant which scales like symplectic area, and for any symplectic embedding $X' \hooksymp X$ we have $\gapac_\bb(X') \leq \gapac_\bb(X)$.
In the case that $X$ is Liouville deformation equivalent to a ball, one can show that $\aug_X\lll \TT \rrr$ is actually a chain homotopy equivalence, so every spectral invariant of $C(X)$ corresponds to some choice of $\bb$. 

Finally, to define the simplified capacities $\gapac_k$, let $\pi_1: \ovl{S}\Q[t] \rightarrow \Q[t]$ denote the projection to tensors of length $1$ (e.g. $t^2 + t^3 \odot t^2 \odot t^5$ maps to $t^2$).
We define
$$\gapac_k(X) :=  \inf_{\bb\;:\; \pi_1(\bb) = t^{k-1}} \gapac_{\bb}(X).$$
In essence, this means we look for the collection of Reeb orbits in $\bdy X$ of minimal action which is closed with respect to the differential of $C(X)$, and which bounds a connected rational curve in $\wh{X}$ satisfying a $\lll \T^{k-1}p\rrr$ constraint (but disregarding any disconnected curves bounded by the same collection).

\subsection{The case of ellipsoids}

To get some intuition for $\gapac_\bb(X)$, we note that when $X$ is an irrational ellipsoid $E(a_1,\dots,a_n)$, the differential on $C(X)$ vanishes for degree parity reasons.
This means that $C(X)$ already agrees with its homology, and the map
$$\aug_X\lll \TT\rrr: C(X) \rightarrow \ovl{S}\Q[t]$$
is in fact an isomorphism. 
Then $\gapac_\bb(X)$ is simply the action of the unique element 
$\left(\aug_X\lll \TT\rrr\right)^{-1}(\bb) \in C(X)$ which corresponds to $\bb$.
However, recall that the map $\aug_X\lll \TT\rrr$ is defined by counting curves in $E(a_1,\dots,a_n)$ satisfying local tangency constraints, so it could be quite nontrivial even in the case $n=2$. 
Indeed, in the very special case of the nearly round ball $E(1,1+\epsilon)$, a closely related problem is to count rational curves in $\mathbb{CP}^2$ satisfying local tangency constraints, which was recently solved in \cite{McSie}.
For other ellipsoids, including those in higher dimensions, and for more general Liouville domains, computing $\gapac_\bb$ seems to involve some very interesting and challenging enumerative problems.

We discuss the computation of the capacities $\gapac_k$ for four-dimensional ellipsoids in \S\ref{subsubsec:some_obstr} below, based on the forthcoming work \cite{McDuffSiegel_in_prep}. As for the larger family of capacities $\gapac_\bb$, a general recursive algorithm for their computation is given in \cite{chscI}, and this will be utilized in the proof of Proposition~\ref{prop:comp}.

\section{Optimal embeddings}\label{sec:computations}

\subsection{The main theorems}
\label{sec:computations1}

We now prove our main results.  To prove Theorem~\ref{thm:main2}, we need a new construction and new obstructions.  These two parts of our argument are logically independent of each other and can be done in either order.  To prove Theorem~\ref{thm:main}, we can use an existing construction and so we just need the obstructions.

\subsubsection{The construction}

We begin with the construction.

\begin{prop}\label{constr} For all $a > 1$ and $S >0$, let $\frac{a}{a+1} \le \mu \le \frac{a}{2}$ and $\lambda = 1 - \frac{\mu}{a}$. There exists a symplectic embedding of $E(a,1,S)$ into an arbitrary neighborhood of $$\{ (z_1, z_2) \mid \pi|z_1|^2 \le \lambda + \mu, \pi|z_2|^2 \le f(\pi|z_1|^2) \} \times \C$$
where $$f(t) = \left\{ \begin{array}{l} 2\lambda - t/2 \quad \mathrm{when} \quad 0 \leq t \leq 2\mu \frac{2\lambda -1}{\lambda + \mu -1};  \\  1 - \frac{(1-\lambda)(t - 2\lambda +1)}{1-\lambda + \mu} \quad \mathrm{when} \quad 2\mu \frac{2\lambda -1}{\lambda + \mu -1} \le t \le \lambda + \mu. \end{array} \right.$$ 
\end{prop}

\begin{remark} Using the work of Pelayo-V\~{u} Ng\d{o}c \cite[Theorem 4.4]{pngoc} we can extend to $S = \infty$ and embed the interior of the ellipsoid into the domain itself, rather than into a neighborhood.
\end{remark}

We defer the proof for a moment, first stating some key corollaries we will need.

\begin{corollary}
\label{cor:emb1}
For any $N \ge 1$ and $a \ge 1$, $1 \le b \le2$ there exists a symplectic embedding
\[ \op{int} E(a,1) \times \C^N \hooksymp \frac{a(b+2)}{(a+1)b} \cdot \left( E(b,1) \times \mathbb{C}^N \right)\]
\end{corollary}

Here, ``int'' denotes the interior.


\begin{proof}[Proof of Corollary~\ref{cor:emb1}]

It clearly suffices to prove this when $N=1$. In Proposition~\ref{constr}, set $\mu = \frac{a}{a+1}$ so $\lambda =  1 - \frac{\mu}{a} = \mu$. In this case $f(t) = 2\lambda - t/2$ for all $0 \le t \le 2\lambda = \lambda + \mu$ and we see that the domain $\{ (z_1, z_2) \mid \pi|z_1|^2 \le \lambda + \mu, \pi|z_2|^2 \le f(\pi|z_1|^2) \}$ is simply $P(2\lambda, 2\lambda) \cap E(4\lambda, 2\lambda)$. This sits inside $E(cb, c)$ when $c \ge \frac{a(b+2)}{(a+1)b}$.

This deals with the case when $a>1$. When $a=1$ we still have an embedding into an arbitrarily small neighborhood, and so can still apply \cite{pngoc} for the precise result.
\end{proof}

\begin{corollary}
\label{cor:emb}
Let $b \in \mathbb{R}_{\ge 2}.$  Then for any $N \ge 1$ and $a \ge b-1$ there exists a symplectic embedding
\[ \op{int} E(a,1) \times \C^N \hooksymp \frac{2a}{a+b-1} \cdot \left( E(b,1) \times \mathbb{C}^N \right)\]
\end{corollary}

\begin{proof}[Proof of Corollary~\ref{cor:emb}]

Note that when $a>1$ we have $\frac{1-\lambda}{1-\lambda + \mu} < \frac{1}{2}$, and so the graph of $f(t)$ is convex. Hence $f(t)$ is bounded above by the linear function between $(0, 2\lambda)$ and $(\lambda + \mu, \lambda)$ and our domain is a subset of $P(\lambda + \mu, 2\lambda) \cap E(2(\lambda + \mu), 2\lambda)$.

In the context of Proposition~\ref{constr}, set $\mu = \frac{a(b-1)}{a+b-1}$. We note that $\frac{a}{a+1} \le \mu \le \frac{a}{2}$ exactly when $2 \le b \le a+1$. Then $\lambda = \frac{a}{a+b-1}$ and we find a symplectic embedding
\begin{align*}
E(a,1) \times \C \hooksymp \left( P\left(\frac{ab}{a+ b-1}, \frac{2a}{a+b-1}\right) \cap E\left(\frac{2ab}{a+b-1}, \frac{2a}{a+b-1}\right) \right) \times \C \\
\subset \frac{2a}{a+b-1} E(b,1) \times \C.
\end{align*}

\end{proof}

We now give the promised proof of the proposition.

\begin{proof}[Proof of Proposition~\ref{constr}.]

Before the proof we fix some notation.

Write $A \subset_{\eps} B$ to mean that the set $A$ lies in an $\eps$ neighborhood of $B$, or $z \in _{\eps} B$ to mean that a point $z$ lies $\eps$ close to $B$.

Let $\pi : \C^3 \to \C$ be the projection onto the $z_1$ plane.

In the $z_1$ plane we fix sets $W_0 = [0,1] \times [0, \mu]$ and $W_i = [2i, 2i+1] \times [0,\lambda]$ for $i \ge 1$.

Finally, $D(a)$ denotes the round closed disk in the plane centered at the origin of area $a$, and $A_i$ are the subsets of the $z_3$ plane given by $A_1 = D(S+\eps)$ and $A_i = D(i(S+\eps)) \setminus D((i-1)(S+\eps))$ for $i \ge 2$.

\begin{proof} The condition $\mu \ge \frac{a}{a+1}$ is equivalent to $\mu \ge 1 - \frac{\mu}{a} = \lambda$, and the condition $\mu \le \frac{a}{2}$ is equivalent to $2\lambda \ge 1$. Both of these inequalities will be used in our construction.

We apply a slightly generalized version of Lemma 2.2 from \cite{hind2015some}. This says that, given $\eps$, there exists a large $K$ and a symplectomorphism $\phi$ from $E(a,1,S)$ to a set $F_K$ with the following properties. For $z \in \C$ we write $F_z = \pi^{-1}(z) \cap F_K$.

\begin{enumerate}
\item $\pi(F_K) \subset_{\eps} \bigcup_{i=1}^K ([2i-1,2i] \times \{0\}) \bigcup_{i=0}^K W_i$;
\item if $z=(u,v) \in_{\eps} W_0$ then $F_z \subset_{\eps} D(1 - \frac{u \mu}{a}) \times A_1$;
\item if $z \in _{\eps} [2i-1,2i] \times \{0\}$ and $i$ is odd, then $F_z \subset_{\eps} D(\lambda) \times A_i$;
\item if $z \in _{\eps} [2i-1,2i] \times \{0\}$ and $i$ is even, then $F_z \subset_{\eps} (D(2\lambda) \setminus D(\lambda))\times A_i$;
\item if $z=(2i+u,v) \in_{\eps} W_i$ and $i$ is odd, then $F_z \subset_{\eps} D((1+u)\lambda)\times(A_i \cup A_{i+1})$;
\item if $z=(2i+u,v) \in_{\eps} W_i$ and $i \ge 2$ is even, then $F_z \subset_{\eps} D((2-u)\lambda) \times(A_i \cup A_{i+1})$.
\end{enumerate}

Apart from slight changes of notation, the modification from Lemma 2.2 consists in increasing the area of $W_0$ (the original lemma fixed $\mu = \lambda = \frac{x}{x+1}$) and a refined description of the fibers over $W_0$. The estimate in item (2) follows easily because $\pi^{-1}(W_0)$ is the set $\{ \pi|z_1|^2 \le \mu \} \subset E(a,1,S)$ and restricted to this set $\phi$ takes the form $\phi(z_1, z_2, z_3) = (\psi(z_1), z_2, z_3)$ where we may assume for all $0 \le u \le 1$ that $\psi$ maps points with $\pi |z_1|^2 \le \mu u$ (outside of which the fiber lies in $\pi|z_2|^2 <   1 - \frac{u \mu}{x}$)  to an $\eps$ neighborhood of the set $[0,u] \times [0, \mu]$. Then if $\psi(z_1) = (u,v)$ we have $\pi |z_1|^2 \ge \mu u - \eps$ and so $\pi|z_2|^2 \le 1- \frac{u \mu}{a} + \eps$.

The next step is to follow Step 3 of the proof from \cite[page 880]{hind2015some} and apply a symplectic immersion $\tau: \pi(F_K) \to \C$. This can be arranged to restrict to an embedding on each of the $W_i$ and each of the intervals $[2i-1,2i] \times \{0\}$, so that the $W_i$ with $i$ odd map into a neighborhood of $[-1,0] \times [0,\lambda]$, the $W_i$ with $i$ even map into $[0,1] \times [0,\mu]$, and the $\eps$ neighborhoods of the intervals $[2i-1,2i] \times \{0\}$ map close to the origin, remaining disjoint from the image of the $W_i$. The condition on $W_i$ with $i$ even is possible since $\lambda \le \mu$.

Let $\iota_{23}$ be the identity map on the $(z_2, z_3)$-plane. Then we note that $(\tau \times \iota_{23}): F_K \to \C^3$ is an embedding. Indeed, the fibers of $\pi$ over $W_i$ and $W_j$ intersect only if $|i-j| \le 1$ (since otherwise by items (5) and (6) their $z_3$ coordinates lie in different $A_k$), and in particular are disjoint if $i$ and $j$ have the same parity. Also the fibers over neighborhoods of different intervals $[2i-1,2i] \times \{0\}$ are disjoint by items (3) and (4).

We refine the immersion $\tau$ slightly to also satisfy the following.

\begin{itemize}\label{immersion}
\item if $z=(2i+u,v) \in W_i$ and $i$ is odd, then $\tau(z) \in_{\eps} [-1+u,0] \times [0,\lambda]$;
\item if $z=(u,v) \in W_0$, then $\tau(z) \in_{\eps} [0,u] \times [0, \mu]$
\item if $z=(2i+u,v) \in W_i$ and $i \ge 2$ is even, then $\tau(z) \in_{\eps} [0,\frac{u \lambda}{\mu}] \times [0, \mu]$.
\end{itemize}

The following describes the fibers of the image of $\tau \times \iota_{23}$.

\begin{lemma}\label{fiberarea} Let $(z_1, z_2, z_3)$ lie in the image of $\tau \times \iota_{23}$ and $z_1 = (u, v)$.

If $-1 \le u \le 0$ then $F_z \subset_{\eps} D((2+u)\lambda)\times \CC$;

if $0 \le u \le \frac{2\lambda -1}{\lambda + \mu -1}$ then $F_z \subset_{\eps} D(2\lambda - u\mu) \times \CC$;

if $\frac{2\lambda -1}{\lambda + \mu -1} \le u \le 1$ then $F_z \subset_{\eps} D(1 - \frac{u \mu}{a}) \times \CC$.
\end{lemma}

\begin{proof} The description of the fibers when $u \le 0$ follows directly from item (5) in the description of $F_K$ and the properties of $\tau$. Also, if $\frac{\lambda}{\mu} \le u \le 1$ then by our description of $\tau$ restricted to the $W_i$ we see that $(u,v)$ is the image of a point in $W_0$, and so the property follows from item (2). (Note that $\frac{\lambda}{\mu} \ge \frac{2\lambda -1}{\lambda + \mu -1}$ because $\lambda <1$ and $\mu \ge \lambda$.)

If $0< u \le \frac{\lambda}{\mu}$ then either $(u,v) = \tau(u',v')$ where $(u', v') \in W_0$ and $u' \ge u$, or $(u,v) = \tau(2i + u',v')$ where $(2i + u', v') \in W_i$ for $i \ge 2$ even and $u' \ge \frac{u \mu}{\lambda}$. In the first case, by item (2), the $z_2$ coordinate of the fiber lies in $D(1 - \frac{u \mu}{a})$ and in the second case, by (6), the $z_2$ coordinate of the fiber lies in $D(2\lambda - u \mu)$. Thus the lemma follows from the fact that $2\lambda - u \mu \ge 1 - \frac{u \mu}{a}$ exactly when $u \le \frac{2\lambda -1}{\lambda + \mu -1}$ (using the assumption that $2\lambda \ge 1$).
\end{proof}

Finally we apply the map $\sigma \times \iota_{23}$, where $\sigma$ is an embedding of a neighborhood of $([-1,0] \times [0,\lambda]) \cup ([0,1] \times [0,\mu])$ in the $z_1$ plane to a neighborhood of the disk $D(\lambda + \mu)$. We can choose $\sigma$ to satisfy the following.

\begin{itemize}
\item if $u \in [-\frac{\mu}{\lambda}t, t]$ and $0 \le t \le \frac{2\lambda -1}{\lambda + \mu -1}$ then $\sigma(u,v) \in_{\eps} D(2t \mu)$ for all $v$;
\item if $u \in [-\frac{2\lambda -1 + (1-\lambda)t}{\lambda}, t]$ and $\frac{2\lambda -1}{\lambda + \mu -1} \le t \le 1$ then $\sigma(u,v) \in_{\eps} D( (2\lambda -1) + (1-\lambda + \mu)t)$ for all $v$.
\end{itemize}

Such a map $\sigma$ exists because the intersection of $([-1,0] \times [0, \lambda]) \cup ([0,1] \times [0, \mu])$, the image of $\tau$, with $\{ u \in [-\frac{\mu}{\lambda}t, t] \}$ has area $2 \mu t$ and the intersection of the image of $\tau$ with $\{ u \in [-\frac{2\lambda -1 + (1-\lambda)t}{\lambda}, t] \}$ has area $(2\lambda -1) + (1-\lambda + \mu)t$. When $t = \frac{2\lambda -1}{\lambda + \mu -1}$ we have that $\frac{\mu}{\lambda}t = \frac{2\lambda -1 + (1-\lambda)t}{\lambda}$ and so  we are imposing a condition on the image of all $(u,v)$.

{\it Claim.} The image of $\sigma \times \iota_{23}$ lies in an $\eps$ neighborhood of $\{ (z_1, z_2) \mid \pi|z_1|^2 \le \lambda + \mu, \pi|z_2|^2 \le f(\pi|z_1|^2) \} \times \C$, concluding the proof.

{\it Proof of the claim.} We check the fibers of $\pi$ over points $w \in D(\lambda + \mu)$. First, if $w$ is in the image of a point in one of the segments $[2i-1,2i] \times \{0\}$ then $w$ is close to $0$ and the $z_2$ coordinate of the fiber lies in $D(2\lambda)$.

Next suppose that $\pi|w|^2 = s + \eps$ where $s \le 2\mu \frac{2\lambda -1}{\lambda + \mu -1}$. Then $w = \sigma(u,v)$ where either $u > \frac{s}{2\mu}$ or $u< -\frac{s}{2\lambda}$ (since by our conditions on $\sigma$ points with $u \in [-\frac{s}{2\lambda}, \frac{s}{2\mu} ]$ are mapped into $D(s)$). By Lemma \ref{fiberarea}, in the first case the $z_2$ coordinate of the fiber lies $\eps$ close to $D(2\lambda - \frac{s}{2})$ and in the second case the $z_2$ coordinate of the fiber also lies in an $\eps$ neighborhood of $D((2 - \frac{s}{2\lambda})\lambda)$. Hence $\pi|z_2|^2 \le 2\lambda - \pi|z_1|^2 /2$.

Finally suppose that $\pi|w|^2 = s + \eps$ where $2\mu \frac{2\lambda -1}{\lambda + \mu -1} \le s \le \lambda + \mu$. Then we see that $w = \sigma(u,v)$ where either $u > \frac{s - (2\lambda -1)}{1-\lambda + \mu}$ or $u< -\frac{ (2\lambda-1)\mu + (1-\lambda)s}{\lambda(1-\lambda + \mu)}$. 
This again follows from our conditions on $\sigma$. Indeed, if $$u \in \left[-\frac{ (2\lambda-1)\mu + (1-\lambda)s}{\lambda(1-\lambda + \mu)}, \frac{s - (2\lambda -1)}{1-\lambda + \mu}\right]$$ then, rewriting, $u \in [-\frac{2\lambda -1 + (1-\lambda)t}{\lambda}, t]$ with $t=\frac{s - (2\lambda -1)}{1-\lambda + \mu}$. The bounds on $s$ imply that $\frac{2\lambda -1}{\lambda + \mu -1} \le t \le 1$ and so by the second bullet point in our description of $\sigma$ points with $u$ in this range are mapped into $D( (2\lambda -1) + (1-\lambda + \mu)t) = D(s)$.

Concluding by Lemma \ref{fiberarea}, if $u > \frac{s - (2\lambda -1)}{1-\lambda + \mu}$ then the $z_2$ coordinate of the fiber lies $\eps$ close to $D(1 -  \frac{s - (2\lambda -1)}{1-\lambda + \mu}\frac{\mu}{a}) = D(1 - \frac{(1-\lambda)(s - 2\lambda +1)}{1-\lambda + \mu})$, recalling that $\lambda = 1-\frac{\mu}{a}$. If $u< -\frac{s}{2\lambda}$ then the $z_2$ coordinate of the fiber lies $\eps$ close to $D(2\lambda - \frac{ (2\lambda-1)\mu + (1-\lambda)s}{1-\lambda + \mu})$ which we check is also $D(1 - \frac{(1-\lambda)(s - 2\lambda +1)}{1-\lambda + \mu})$. Hence $\pi|z_2|^2 \le 1 - \frac{(1-\lambda)(\pi|z_1|^2 - 2\lambda +1)}{1-\lambda + \mu} + \eps$.
\end{proof}

With the claim proven, we have completed the proof of the proposition.
\end{proof}

\subsubsection{Some obstructions}\label{subsubsec:some_obstr}

We now turn our attention to the obstructive side.  Notably, this will be quite short, because we can cite work on these higher capacities that has previously been done or is forthcoming.  Namely, here we only recall the following computations for the capacities of ellipsoids and polydisks from \cite[\S 6.3]{hsc}:
\begin{align}
&\gapac_k(P(1,a)) = \min(k,a+\lceil \tfrac{k-1}{2} \rceil)&\text{for}\; a \geq 1,\;k \geq 1\;\text{odd}\label{eq:poly_comp}\\
&\gapac_k(E(1,a)) = k &\text{for}\; a \geq 1,\;1 \leq k \leq a.\label{eq:ellip_comp}
\end{align}
It seems plausible that the computation for $P(1,a)$ is also valid for $k$ even. This would follow if we knew that the capacities $\gapac_k$ are nondecreasing with $k$, although this is not yet clear. 

We will also need the following more general expected formula for ellipsoids,
which will be proved in \cite{McDuffSiegel_in_prep}.
For $1 \leq a \leq 3/2$, we have 
\begin{align}
\gapac_k(E(1,a)) = 
\begin{cases}
1 + ia&\text{ for } k = 1+3i\text{ with } i \geq 0\\
a+ia&\text{ for } k = 2+3i\text{ with } i \geq 0\\
2 + ia& \text{ for } k = 3+3i\text{ with } i \geq 0.
\end{cases} \label{eq:ellip_comp3}
\end{align}
For $a > 3/2$, we have
\begin{align}
\gapac_k(E(1,a)) = 
\begin{cases}
k&\text{ for } 1 \leq k \leq \lfloor a \rfloor\\
a + i&\text{ for } k = \lceil a \rceil + 2i \text{ with } i \geq 0\\
\lceil a \rceil + i&\text{ for } k = \lceil a \rceil + 2i+1 \text{ with } i \geq 0.
\end{cases} \label{eq:ellip_comp2}
\end{align}
\subsubsection{The proofs}

We now give the promised proofs.

\begin{proof}[Proof of Theorem~\ref{thm:main2}]
Let $a, b$ and $N$ be as in the statement of the theorem.  Then, by Corollary~\ref{cor:emb}, we have
\[ c^N_{b,ell}(a) \le \frac{2a}{a+b-1}.\]
To prove the opposite inequality, we use the higher capacities $\gapac_k.$  That is, take $k = a$.  Then, by \eqref{eq:ellip_comp} and  \eqref{eq:ellip_comp2}, we have,
\[ \gapac_k(E(1,a)) = a, \quad \gapac_k(E(1,b)) = \frac{a+b-1}{2}.\]
Hence, by the scaling, monotonicity, and stabilization properties of the $\gapac_k$ in Theorem~\ref{thm:main_props},  
we have
\[ c^N_{b,ell}(a) \ge \frac{2a}{a+b-1},\]
hence the theorem.
\end{proof}

\begin{remark}
Note that in the above proof we only need the inequality ${\gapac_a(E(1,b)) \leq \tfrac{a+b-1}{2}}$, and in the case that $b$ is even (and hence $a \geq b+1$ is odd) this can be deduced directly from \eqref{eq:poly_comp}.
Indeed, by \eqref{eqn:equiv}
there is an embedding $E(1,b) \hooksymp P(1,b/2)$, whence we have $$\gapac_a(E(1,b)) \leq \gapac_a(P(1,b/2)) = b/2 + \lceil (a-1)/2\rceil = \frac{a+b-1}{2}.$$
\end{remark}






\begin{proof}[Proof of Theorem~\ref{thm:main}]

The proof is similar to the previous one.  Let $a, b$ and $N$ be as in the statement of the theorem.

The bound
\[ c^N_{b,poly}(a) \le \frac{2a}{a+2b-1}\] 
follows from the existence of a variant of the embedding from above, which was previously shown to exist in \cite[Lem. 1.3]{cristofaro2017symplectic}.  

To show that no better embedding exists, we use the above capacities.  Namely, let $k = a$.  Then, by ~\eqref{eq:poly_comp} and ~\eqref{eq:ellip_comp} above, we have
\[ \gapac_k(E(1,a)) = a, \quad \quad \gapac_k(P(1,b)) = b + \frac{a-1}{2}.\]
The theorem now follows by the same argument as above.
\end{proof}

\begin{example}
\label{ex:mcd}
It is interesting to compare the above methods with the case $b = 1$.  For this, we recall for the convenience of the reader an argument from \cite[\S 1.4]{hsc}.  There, a variant of the embedding used in the previous theorems, constructed in \cite{hind2015some}, gives
\[ c^N_{1,ell}(a) \le \frac{3a}{a+1}.\]
On the other hand, if $a$ is an integer congruent to two, modulo three, then taking $k = a$ as above yields
\[\gapac_k(E(1,a) \times \mathbb{C}^n) = a, \quad \quad \gapac_k(E(1,1) \times \mathbb{C}^n) = \frac{1+a}{3}.\]
Hence, combining these inequalities, we get that for $a$ congruent to two modulo three,
\[ c^N_{1,ell}(a) = \frac{3a}{a+1}.\]
This recovers the result of McDuff \cite[Thm. 1.1]{Mcduff_remark}.
\end{example}


\subsection{The rescaled embedding function}
\label{sec:cor}
We now provide the proofs of the promised corollaries regarding the conjecture of the second named author, Frenkel, and Schlenk.

\begin{proof}[Proof of Corollary~\ref{cor:main}]

We will first prove the statement about $c^N_{b,poly}$, after which the result about $c^N_{b,ell}$ will follow easily.
 
The function $c^N_{b,poly}(a)$ is nonincreasing in $N$.  We want to show that it is in fact constant in $N$ for $a$ in the intervals given by the theorem.  The computation of $c^0_{b,poly}(a)$ from \cite{cristofaro2017symplectic}, together with Theorem~\ref{thm:main} from above, shows that it does not depend on $N$ for the exterior (middle) corner of each linear step.
 
Now note that if an embedding
\[ E(1,a) \times \mathbb{C}^n \hooksymp \lambda P(1,b) \times \mathbb{C}^n \]
exists, then for any $a' > a$, by scaling there is an embedding
\[ E(1,a') \times \mathbb{C}^n  \hooksymp \frac{a'}{a} \lambda P(1,b) \times \mathbb{C}^n.\]
Thus, $c^N_{b,poly}(a') \le \frac{a'}{a} c^N_{b,poly}(a).$  So, given $y_0 = c^N(a)$, the graph of $c^N(a')$ for $a' > a$ cannot lie above the line through $(a,y_0)$ and the origin.  For future reference, we call this the {\bf subscaling property}.  We can now prove the corollary.  

Consider any linear step for $c^0_{b,poly}(a)$.  Recall that this consists of a linear part, then an exterior corner, and then a horizontal part.  Consider the linear part.  We want to show that this stabilizes.  We know that $c^N_{b,poly}(a) \le c^0_{b,poly}(a)$.  If there were any $a$ value for which strict inequality held, then by the linearity property above, at the exterior corner $a_0$ of the step, we would have $c^N_{b,poly}(a_0) < c^0_{b,poly}(a_0).$ However, above we saw in Theorem~\ref{thm:main} that the exterior corner is stable.  Hence, the whole linear part must stabilize.  As for the horizontal part, we know that we must have $c^N_{b,poly} \le c^0_{b,poly}$, but on the other hand the function $c^N_{b,poly}$ is nondecreasing, and so must be constant here.  Thus, the whole step stabilizes, so all the linear steps do.

In view of Theorem~\ref{thm:main2}, the exact same argument implies the result about $c^N_{2b,ell}$, since for $N = 0$ there is an equivalence of embeddings \eqref{eqn:equiv}. \qedhere
 
\end{proof}

\begin{proof}[Proof of Corollary~\ref{cor:rescal}]

Corollary~\ref{cor:main} shows that, after the initial part of the graph, where $c_{b,poly}^N(a) = 1$, the graph has $\lceil \sqrt{2b} \rceil + 1$ linear steps that are all stable.  The length of these steps is given by the formula $\ell_b(k)$ from \cite[p. 6]{cristofaro2017symplectic}.  In particular, as explained there, the length of the $k^{th}$ step converges to $2$ as $b$ tends to infinity.  Since the steps are centered at the odd numbers, increase in number without bound as $b$ increases, and our rescaled function is centered so that the initial part of the graph with height one, that is, the part determined by Gromov's nonsqueezing theorem, does not appear, the result follows.  \qedhere


\end{proof}

\subsection{The first step}

We now prove Proposition~\ref{prop:main}.

\begin{proof}[Proof of Proposition~\ref{prop:main}.]

The key is the following lemma.

\begin{lemma}
\label{lem:nice}
Let $a_0$ be the smallest odd integer that is no less than $2b-1$.  There is a symplectic embedding
\begin{equation}
\label{eqn:needed}
\op{int}\left(E\left(1,\frac{a_0-1}{2} + b\right)\right) \hooksymp P(1,b).
\end{equation}
\end{lemma}

\begin{proof}

We first explain why it suffices to prove the lemma for $b$ rational.  Given an irrational $b$, we can choose rational numbers $b_n$ converging to $b$ from below.  Then, if the lemma is true for each $b_n$ and the $b_n$ are sufficiently close to $b$, composing with the inclusion $P(1,b_n) \subset P(1,b)$ gives embeddings $\op{int}\left(E\left(1,\frac{a_0-1}{2} + b_n\right)\right) \hooksymp P(1,b),$
hence the desired embedding \eqref{eqn:needed} by \cite[Cor. 1.6]{dcg}.  

We thus henceforth assume that $b$ is rational. Then, by for example \cite[Thm. 2.1] {dcg}, it is equivalent to find an embedding
\begin{equation}
\label{eqn:need}
\op{int}\left(E(1,b)\right) \cup \op{int} \left(E\left(1, \frac{a_0-1}{2} \right)\right) \hooksymp P(1,b).
\end{equation}
Indeed, the argument for \cite[Thm. 2.1]{dcg} implies that both \eqref{eqn:needed} and \eqref{eqn:need} are equivalent to ball packing problems of the $P(1,b)$, where in the first case, the size of the balls is given by the weight sequence defined in \cite[\S 2]{dcg} for $(a_0-1)/2 + b$, and in the second case the size of the balls is given by the union of the weight sequence for $b$ and for $(a_0-1)/2$.   Since $(a_0-1)/2$ is an integer, the first $(a_0-1)/2$ of the weights for $(a_0-1)/2 + b$ will be $1$, so \eqref{eqn:needed} and \eqref{eqn:need} are equivalent to the same ball packing problem.  

We know that $a_0 \le 2b+1$, hence
\begin{equation}
\label{eqn:bound}
 \frac{a_0-1}{2} \le b.
 \end{equation}
We can therefore find an embedding as in \eqref{eqn:need} as follows.  We think of the moment image of $P(1,b)$ as a union of two triangles, joined along the diagonal that does not contain the origin.  The triangle with legs on the axes contains an $E(1,b)$ factor by inclusion.  As for the other triangle, it is affine equivalent to the first, via multiplication by $-I_2$, where $I_2$ is the two-by-two identity matrix.  Hence, by the Traynor trick, see for example \cite{traynor} and \cite[Lem. 1.8]{cgetal}, it also contains a copy of an $\op{int}(E(1,b))$, disjoint from the interior of the first $E(1,b)$.  Now, by \eqref{eqn:bound} this latter $\op{int}(E(1,b))$ contains a copy of $\op{int}\left(E(1,(a_0-1)/2)\right).$
\end{proof}

We can now prove the proposition.  We first prove the second bullet point.  By Lemma~\ref{lem:nice}, we know that $c^N_{b,poly} \le 1$, for $a$ in the given range.  However, by Gromov's non-squeezing theorem, we also know that $c^N_{b,poly} \ge 1$, for $a$ in this range.  As for the rest of the second bullet point, this follows from the 
subscaling property of $c^N_{b,poly}$, as in the proof of Corollary~\ref{cor:main} above, given the lower bound on $c^{N}_{b,poly}(a_0)$ coming from Theorem~\ref{thm:main}.

We now prove the first bullet point.  The result for $1 \le a \le b$ follows because inclusion gives an embedding for $a$ in this range, which is optimal by Gromov's nonsqueezing theorem.  Similarly, for $b \le a \le \lfloor b \rfloor+1$, scaling gives an embedding as in the subscaling property, which is optimal by the $(\lfloor b \rfloor+1)^{st}$ Ekeland-Hofer capacity, see eg \cite[\S 2.3.1, \S 4.1.1]{q} for the relevant formula.


\end{proof}

\subsection{The other parity}

The proof of the remaining proposition, Proposition~\ref{prop:comp}, requires the $\gapac_b$ and computer assistance as well.
It turns out that the simplified capacities $\gapac_k$ do not suffice in these cases.
For example, for $E(1,6) \times \C^N \hooksymp \lambda \cdot P(1,1) \times \C^N$, one can check that the simplified capacities give only $\lambda \geq 5/3$, whereas we have in fact $c_{1,poly}^N(6) = 12/7$ for $N \in \Z_{\geq 1}$.

On the other hand, we have the more general capacities $\gapac_\bb$, which could in principle give sharp obstructions for all $a \in \R_{\geq 1}$ and $b \in \Z_{\geq 1}$ in \eqref{eqn:stabilizedell} and \eqref{eqn:stabilizedone}.
This is related to the discussion at the end of \cite[\S 6.3]{hsc}, where it is observed that the simplified capacities $\gapac_k$ do not generally give sharp obstructions for $E(1,a) \times \C^N \hooksymp \lambda \cdot E(1,1) \times \C^N$, but the capacities $\gapac_\bb$ necessarily give sharp obstructions at least for $a \leq \tau^4$.
Moreover, the formalism from \cite{chscI} gives an explicit recursive algorithm to compute the capacities $\gapac_\bb$ for all convex toric domains, although unfortunately it appears to be somewhat difficult to compute with ``by hand''.

\begin{proof}[Proof of Proposition~\ref{prop:comp}]

We begin with the computation of $c^N_{1,poly}(a)$ for $a = 6,8,\dots,\sixty$. By \cite{hind2015some}, we have the upper bound $c^N_{1,poly}(a) \leq \tfrac{2a}{a+1}$, so it suffices to establish the lower bound $c^N_{1,poly}(a) \geq \tfrac{2a}{a+1}$.
Suppose that we have a symplectic embedding $E(1,a)\times \C^N \hooksymp \lambda \cdot P(1,1) \times \C^N$.

Following the notation and exposition of \cite{chscI}, the idea is as follows.
By \cite[Cor. 1.2.3]{chscI}, there is a filtered $\Li$ homomorphism $Q: V_{P(\lambda,\lambda)} \rightarrow V_{E(1,a)}$ which is unfiltered $\Li$ homotopic to the identity.
Here $V$ is an explicit DGLA with generators $\alpha_{i,j}$ for $i,j \in \Z_{\geq 1}$ and $\beta_{i,j}$ for $i,j \in \Z_{\geq 0}$ not both zero.
The filtered DGLA $V_{P(\lambda,\lambda)}$ is just $V$ as an unfiltered DGLA, and its filtration is specified by
$$\calA_{P(\lambda,\lambda)}(\alpha_{i,j}) = \calA_{P(\lambda,\lambda)}(\beta_{i,j}) = \lambda i + \lambda j.$$
Similarly, the filtered DGLA $V_{E(1,a)}$ is just $V$ as an unfiltered DGLA, with filtration specified by
$$\calA_{E(1,a)}(\alpha_{i,j}) = \calA_{E(1,a)}(\beta_{i,j}) = \max(i,aj).$$
Recall that an $\Li$ homomorphism $Q: V_{P(\lambda,\lambda)} \rightarrow V_{E(1,a)}$ consists of a sequence of maps $Q^l: \odot^l V_{P(\lambda,\lambda)} \rightarrow V_{E(1,a)}$ for $l = 1,2,3,\dots$, and these must satisfy an infinite sequence of certain quadratic relations.

Any element of the form $\beta_{i_1,j_1} \odot \dots \odot \beta_{i_k,j_k}$ defines a cycle in the bar complex $\bar V_{P(\lambda,\lambda)}$. 
In particular, $\wh{Q}(\beta_{i_1,j_1} \odot \dots \odot \beta_{i_k,j_k})$ must be homologous to $\beta_{i_1,j_1} \odot \dots \odot \beta_{i_k,j_k}$ in $\bar V_{E(1,a)}$.
Moreover, there is a filtered $\Li$ homomorphism $\Phi_{1,a}: V_{E(1,a)} \rightarrow V^\can_{E(1,a)}$, where $V^\can_{E(1,a)}$ denotes the homology of $V_{E(1,a)}$ (viewed as a filtered $\Li$ algebra with trivial $\Li$ operations), and hence 
$(\wh{\Phi}_{1,a} \circ \wh{Q})(\beta_{i_1,j_1} \odot \dots \odot \beta_{i_k,j_k})$ is homologous to $\wh{\Phi}_{1,a}(\beta_{i_1,j_1} \odot \dots \odot \beta_{i_k,j_k})$ in $\bar V^{\can}_{E(1,a)}$. 

Now suppose that we have $a = p/q$ with $p+q = 2d$ for some $p,q,d \in \Z_{\geq 1}$.
Consider some $d_1,d_2 \in \Z_{\geq 0}$ satisfying $d_1 + d_2 = d$, and suppose that we have 
\begin{align}\label{eqn:str_coeff_nonvan}
\Phi_{1,a}^d(\odot^{d_1}\beta_{1,0} \odot \odot^{d_2}\beta_{0,1}) \neq 0.
\end{align}
Then we claim that we have $\lambda \geq \tfrac{2a}{a+1}$, which gives the desired lower bound.
Indeed, for a general input of the form $\beta_{i_1,j_1} \odot \dots \odot \beta_{i_k,j_k}$, it follows by degree considerations that  $\Phi_{1,a}^k(\beta_{i_1,j_1} \odot \dots \odot \beta_{i_k,j_k})$ is either trivial, or else it is the unique element up to scaling in $V^{\can}_{E(1,a)}$ of its given degree. In the latter case, its action is given by the $l$th Ekeland--Hofer capacity of $E(1,a)$, i.e. $c_l^{\op{EH}}(E(1,a))$, for $l = \sum_{m=1}^k (i_m + j_m) + k-1$.
Also, the action of the input is given by 
$$\calA_{P(\lambda,\lambda)}(\beta_{i_1,j_1} \odot \dots \odot \beta_{i_k,j_k}) = \sum_{m=1}^k \calA_{P(\lambda,\lambda)}(\beta_{i_m,j_m}) = \sum_{m=1}^k (\lambda i_m + \lambda 
j_m).$$

Specializing to the case of input $\odot^{d_1}\beta_{1,0} \odot \odot^{d_2}\beta_{0,1}$ and $l = 2d-1$,
using $a = p/q$ and $p+q = 2d$ it is straightforward to check that we have $c_l^{\op{EH}}(E(1,a)) = p$. 
Since $\wh{\Phi}_{1,a} \circ \wh{Q}$ is filtration-preserving
and $\Phi^d(\odot^{d_1}\beta_{1,0} \odot \odot^{d_2}\beta_{0,1})$ is a summand of the image of $[\odot^{d_1}\beta_{1,0} \odot \odot^{d_2}\beta_{0,1}]$ under $[\wh{\Phi}_{1,a} \circ \wh{Q}]$, we must have
$\lambda (d_1+d_2) \geq p$, and hence
$$\lambda \geq \frac{p}{d} = \frac{2p}{p+q} = \frac{2a}{a+1},$$
as claimed.

Let us now specialize to the case that $a$ is an even integer. Then we have $a = p/q$ for $p = 2a$ and $q = 2$, and hence $p+q = 2d$ for $d = a+1$. By computer calculations, \eqref{eqn:str_coeff_nonvan} holds for $d_1 = 3$ and $d_2 = d - d_1 = a - 2$ for $a = 6,\dots,\sixty$.
Geometrically, this corresponds to a nonvanishing count of rational curves in $\CP^1 \times \CP^1 \setminus \tfrac{1}{\lambda} \cdot E(1,a)$ of bidegree $(d_1,d_2)$ with one negative puncture asymptotic to the $p = 2a$ fold cover of the short simple Reeb orbit. 
Curiously, the analogous counts for $d_1 = 1,2$ vanish.

\medskip

The computation of $c^N_{2,ell}(a)$ for $a = 6,8,\dots,\sixty$ is similar. 
In this case, we suppose that we have a symplectic embedding $E(1,a)\times \C^N \hooksymp \lambda \cdot E(1,2) \times \C^N$,
and we take our input cycle to be of the form $\odot^3 \beta_{2,1} \odot \odot^{d-3}\beta_{1,0}$, for $d = a-2$.
By computer calculation we have 
\begin{align}\label{eqn:str_coeff_nonvan2}
\Phi_{1,a}^d(\odot^{3}\beta_{2,1} \odot \odot^{d-3}\beta_{0,1}) \neq 0
\end{align}
for $a = 6,8,\dots,\sixty$.
The action of the output is that of the $l$th Ekeland--Hofer capacity of $E(1,a)$ for $l = 5 + 2d$, and we have $c_l^{\op{EH}}(E(1,a)) = 2a$.
Meanwhile, the action of the input is
$$\calA_{E(1,2)}(\odot^3 \beta_{2,1} \odot \odot^{d-3}\beta_{1,0}) = 6 + (d-3) = a+1,$$
whence the lower bound $\lambda \geq \frac{2a}{a+1}$ readily follows.
\end{proof}

\section{Discussion}
\label{sec:discuss}

We close by discussing some natural follow-up questions to our work. 

\subsection{Beyond the rescaled function}


One can of course ask whether the function $c_{b,poly}^N(a)$ can in any sense be computed completely.  
 As explained in \cite[Lem. 1.3]{cristofaro2017symplectic}, and mentioned previously here, a previous folding construction of the second named author gives the bound
\[ c_{b,poly}^N(a) \le \frac{2a}{a+2b-1}.\]
This bound can not be optimal for all $a$.  For example, as we have seen in this paper, there are sometimes four-dimensional embeddings beating this bound, and these can be stabilized by taking the product with the identity.  For $a$ sufficiently large with respect to $b$, though, in particular for 
\begin{equation}
\label{eqn:switch}
a \ge ( \sqrt{2b}+1)^2,
\end{equation}
the above folding bound beats the four-dimensional volume obstruction, and so must give a better construction than any stabilized four-dimensional one.
The main question at the moment here is as follows.


\begin{question}
\label{conj:ten}
Is it the case that either $c^N_{b,poly}(a) = c_{b,poly}^0(a),$ or
\[ c^N_{b,poly}(a) = \frac{2a}{a+2b-1}?\]
\end{question}

If this is true, it looks hard to prove.  For example, if $a < (\sqrt{2b} + 1)^2$, then the volume bound is strictly below the folding bound from above.  On the other hand, for $b \in \mathbb{Z}_{\ge 2}$, it is known that there are entire intervals of the subset $a < (\sqrt{2b} + 1)^2$ for which the volume bound is optimal for $c^0_{b,poly}$: for example, for $b = 2$,  \cite[Thm. 1.1]{cristofaro2017symplectic} states that there is an interval on which $c^0_{b,poly}$ is given by the volume starting at $a = 7.84$, but  on the other hand by \eqref{eqn:switch} the folding curve is above the volume curve up until $a = 9$.  Finding the holomorphic curves needed to show that this volume bound stabilizes would be a completely new phenomenon. 

The same question, but concerning $c^N_{b,ell}$ is also open and just as interesting.

\subsection{The opposite parity}

It is also natural to ask what happens for the stabilized embedding problem for ellipsoids, when the parity of the domain and target are the same. For example, one might hope that an analogue of our Proposition~\ref{prop:comp} holds in the case $b > 2$.  If this is true, however, it is not so clear how to prove it: our preliminary computer search to generalize the method required to prove it has not turned up promising candidates.  It would be very interesting to find a candidate of curves to solve this problem, or to find another embedding.

\subsection{The region from $b = 1$ to $b = 2$}
\label{sec:b12}

For $b \ge 2$, our Corollary~\ref{cor:emb} produces an embedding such that 
\[ c^N_{b,ell}(a) \le \frac{2a}{a + b -1}.\]
Meanwhile for $1 \le b \le 2$, Corollary~\ref{cor:emb1} shows
\begin{equation}
\label{eqn:possible} 
c^N_{b,ell}(a) \le \frac{a(b+2)}{(a + 1)b}.
\end{equation}
It is interesting to ask when this bound is sharp, for instance whether there are sequences of $a$ where this holds. We now list some facts suggesting the answer may not be straightforward.

Note that when $b=1$ the bound gives $$c^N_{1,ell}(a) \le \frac{3a}{a + 1},$$ which as mentioned above is sharp when $a \equiv 2$ modulo $3$, \cite{Mcduff_remark}. There is another sequence starting at $a=2$ where \eqref{eqn:possible} is an equality. By work of the first and second named authors,  \cite{CG-Hind_products}, we have $c^N_{1,ell}(a) = c^0_{1,ell}(a)$ for all $1 \le a \le \tau^4$. This region of the graph is an infinite staircase, that is, piecewise linear with infinitely many singular points accumulating at $\tau^4$, see \cite{mcduff2012embedding}. Between these singular points the graph alternates between being constant and sitting on a line through the origin. One can check the corners of the stairs, the left endpoints of the constant intervals, lie on the folding graph $\frac{3a}{a+1}$.

When $b=2$ our bound gives $$c^N_{2,ell}(a) \le \frac{2a}{a + 1}.$$ The graph of $c^0_{2,ell}$ also begins with an infinite staircase, see \cite{cgk, fm}, and again the tips of the stairs lie on the graph $\frac{2a}{a + 1}$. It seems extremely likely that at such $a$ we have $c^N_{2,ell}(a) = c^0_{2,ell}(a)$ for all $N$ so the bound \eqref{eqn:possible} is again sharp.

However when $b=3/2$ the situation is mysterious. Now our bound gives $$c^N_{3/2,ell}(a) \le \frac{7}{3}\frac{a}{a + 1}.$$ Here again work of the first named author and Kleinman shows that $c^0_{3/2,ell}(a)$ has an infinite staircase \cite{cgk}, but now the tips of the stairs lie on the graph $\frac{2a}{a+1}$. Moreover the $\gapac_k$ show that $c^N_{3/2,ell}(a) \ge \frac{2a}{a+1}$ at integer $a$. It is unclear whether an improved construction can show this lower bound is indeed sharp, or whether enhanced obstructions can be used to show that even though the folding graph \eqref{eqn:possible} lies strictly above the infinite staircase it is still asymptotically sharp.

\subsection{A combinatorial rule?}

While the functions $c^0_{b,ell}$ and $c^0_{b,poly}$ themselves are known to be quite complicated (see for example \cite{mcduff2012embedding, usher2019infinite}), they are governed by simple to state combinatorial rules.  For example, McDuff shows in \cite{mcduff2011hofer} that $c^0_{b,ell}$ is completely determined by the combinatorics of the sequence $N(a,b)$, whose $k^{th}$ term is the $(k+1)^{st}$ smallest entry among the nonnegative integer linear combinations of $a$ and $b$.    It would be extremely interesting if the functions $c^N_{b,ell}$ and $c^N_{b,poly}$ are also governed by some kind of relatively simple to state combinatorial rule.  It might be easier to find such a rule than to actually compute these functions explicitly.

\bibliographystyle{math}
\bibliography{stabilized_ellipsoid_integral}

\newcommand{\etalchar}[1]{$^{#1}$}
\begin{thebibliography}{CCGF{\etalchar{+}}}

\bibitem[BH]{bao2015semi}
Erkao Bao and Ko~Honda.
\newblock {Semi-global Kuranishi charts and the definition of contact
  homology}.
\newblock {\em arXiv:1512.00580} (2015).

\bibitem[CCGF{\etalchar{+}}]{cgetal}
Keon Choi, Dan Cristofaro-Gardiner, David Frenkel, Michael Hutchings, and
  Vinicius Ramos.
\newblock {Symplectic embeddings into four-dimensional concave toric domains}.
\newblock {\em J. Topol.} {\bf 7\;}(2014), 1054--1076.

\bibitem[CHLS]{q}
Kai Cieliebak, Helmut Hofer, Janko Latschev, and Felix Schlenk.
\newblock {Quantitative symplectic geometry}.
\newblock {\em Dynamics, Ergodic Theory, and Geometry: Dedicated to Anatole
  Katok} {\bf 54\;}(2007), 1--44.

\bibitem[CG]{dcg}
Dan Cristofaro-Gardiner.
\newblock {Symplectic embeddings from concave toric domains into convex ones}.
\newblock {\em J. Diff. Geom.} {\bf 112\;}(2019), 199--232.

\bibitem[CGK]{cgk}
Dan Cristofaro-Gardiner and Aaron Kleinman.
\newblock {Ehrhart functions and symplectic embeddings of ellipsoids}.
\newblock {\em Journal of the London Mathematical Society} {\bf 101\;}(2020),
  1090--1111.

\bibitem[CGFS]{cristofaro2017symplectic}
Daniel Cristofaro-Gardiner, David Frenkel, and Felix Schlenk.
\newblock {Symplectic embeddings of four-dimensional ellipsoids into integral
  polydiscs}.
\newblock {\em Algebraic \& Geometric Topology} {\bf 17\;}(2017), 1189--1260.

\bibitem[CGH]{CG-Hind_products}
Daniel Cristofaro-Gardiner and Richard Hind.
\newblock {Symplectic embeddings of products}.
\newblock {\em Comment. Math. Helv.} {\bf 93\;}(2018), 1--32.

\bibitem[CGHM]{Ghost_stairs_stabilize}
Daniel Cristofaro-Gardiner, Richard Hind, and Dusa McDuff.
\newblock {The ghost stairs stabilize to sharp symplectic embedding
  obstructions}.
\newblock {\em J. Topol.} {\bf 11\;}(2018), 309--378.

\bibitem[EGH]{eliashberg2000introduction}
Yakov Eliashberg, A~Givental, and Helmut Hofer.
\newblock {Introduction to symplectic field theory}.
\newblock {\em Visions in mathematics: GAFA 2000 Special volume} (2000),
  560--673.

\bibitem[FH]{fish2018lectures}
Joel~W Fish and Helmut Hofer.
\newblock {Lectures on polyfolds and symplectic field theory}.
\newblock {\em arXiv:1808.07147} (2018).

\bibitem[FM]{fm}
David Frenkel and Dorothee M\"uller.
\newblock {Symplectic embeddings of four-dimensional ellipsoids into cubes}.
\newblock {\em Journal of Symplectic Geometry} {\bf 13\;}(2015), 765--847.

\bibitem[Hin1]{hind2015some}
Richard Hind.
\newblock {Some optimal embeddings of symplectic ellipsoids}.
\newblock {\em Journal of Topology} {\bf 8\;}(2015), 871--883.

\bibitem[Hin2]{hist}
Richard Hind.
\newblock {Stabilized symplectic embeddings}.
\newblock {\em Complex and Symplectic Geometry} (2017), 85--93.

\bibitem[HK1]{Hind-Kerman_new_obstructions}
Richard Hind and Ely Kerman.
\newblock {New obstructions to symplectic embeddings}.
\newblock {\em Invent. Math.} {\bf 196\;}(2014), 383--452.

\bibitem[HK2]{Hind-Kerman_correction}
Richard Hind and Ely Kerman.
\newblock {Correction to: {N}ew obstructions to symplectic embeddings}.
\newblock {\em Invent. Math.} {\bf 214\;}(2018), 1023--1029.

\bibitem[HWZ]{hofer2017polyfold}
Helmut Hofer, Krzysztof Wysocki, and Eduard Zehnder.
\newblock {\em Polyfold and {F}redholm theory}.
\newblock Springer, 2021.

\bibitem[Hut]{hutchings_2013}
Michael Hutchings.
\newblock {Rational SFT using only q variables}.
\newblock
  https://floerhomology.wordpress.com/2013/04/23/rational-sft-using-only-q-variables,
  2013.

\bibitem[HN]{HuN}
Michael Hutchings and Jo~Nelson.
\newblock {Cylindrical contact homology for dynamically convex contact forms in
  three dimensions}.
\newblock {\em J. Symplectic Geom.} {\bf 14\;}(2016), 983--1012.

\bibitem[Ish]{ishikawa2018construction}
Suguru Ishikawa.
\newblock {Construction of general symplectic field theory}.
\newblock {\em arXiv:1807.09455} (2018).

\bibitem[McD1]{mcduff2011hofer}
Dusa McDuff.
\newblock {The Hofer conjecture on embedding symplectic ellipsoids}.
\newblock {\em J. Diff. Geom.} {\bf 88\;}(2011), 519--532.

\bibitem[McD2]{Mcduff_remark}
Dusa McDuff.
\newblock {A remark on the stabilized symplectic embedding problem for
  ellipsoids}.
\newblock {\em Eur. J. Math.} {\bf 4\;}(2018), 356--371.

\bibitem[MS1]{McDuff-Salamon}
Dusa McDuff and Dietmar Salamon.
\newblock {Introduction to symplectic topology (third edition)}.
\newblock {\em Oxford Graduate Texts in Mathematics} (2017).

\bibitem[MS2]{mcduff2012embedding}
Dusa McDuff and Felix Schlenk.
\newblock {The embedding capacity of 4-dimensional symplectic ellipsoids}.
\newblock {\em Annals of Mathematics} {\bf 175\;}(2012), 1191--1282.

\bibitem[MS3]{McSie}
Dusa McDuff and Kyler Siegel.
\newblock {Counting curves with local tangency constraints}.
\newblock {\em {\em To appear} in Journal of Topology} (2021).

\bibitem[MS4]{McDuffSiegel_in_prep}
Dusa McDuff and Kyler Siegel.
\newblock {Symplectic capacities, unperturbed curves, and convex toric
  domains}.
\newblock (In preparation).

\bibitem[Par]{pardon2019contact}
John Pardon.
\newblock {Contact homology and virtual fundamental cycles}.
\newblock {\em Journal of the American Mathematical Society} {\bf 32\;}(2019),
  825--919.

\bibitem[PVuN]{pngoc}
\'{A}lvaro Pelayo and San V\~{u}~Ng\d{o}c.
\newblock {Hofer's question on intermediate symplectic capacities}.
\newblock {\em Proceedings of the London Mathematical Society} {\bf
  110\;}(2015), 787--804.

\bibitem[Sie1]{hsc}
Kyler Siegel.
\newblock {Higher symplectic capacities}.
\newblock {\em arXiv:1902.01490} (2019).

\bibitem[Sie2]{chscI}
Kyler Siegel.
\newblock {Computing higher symplectic capacities I}.
\newblock {\em {\em To appear in} International Mathematics Research Notices}.

\bibitem[Sie3]{chscII}
Kyler Siegel.
\newblock {Computing higher symplectic capacities II}.
\newblock (In preparation).

\bibitem[Tra]{traynor}
Lisa Traynor.
\newblock {Symplectic packing constructions}.
\newblock {\em J. Diff. Geom.} {\bf 42\;}(1995), 411--429.

\bibitem[Ush]{usher2019infinite}
Michael Usher.
\newblock {Infinite staircases in the symplectic embedding problem for
  four-dimensional ellipsoids into polydisks}.
\newblock {\em Algebraic \& Geometric Topology} {\bf 19\;}(2019), 1935--2022.

\end{thebibliography}

\end{document}